\documentclass[reqno]{amsart}
\usepackage{amsthm, amscd, amsfonts,graphicx}
\usepackage{color}
\usepackage{enumerate}
\usepackage{verbatim}
\numberwithin{equation}{section}
\theoremstyle{plain}
 \newtheorem{theorem}{Theorem}[section]
 \newtheorem{lemma}[theorem]{Lemma}

 \newtheorem{corollary}[theorem]{Corollary}
\theoremstyle{definition}

\theoremstyle{remark}

\newenvironment{enumeratei}{\begin{enumerate}[\quad\upshape (i)]} {\end{enumerate}}

\newcommand \ts {$(T,S)$}
\newcommand \chain[1] {\mathsf C_{#1}} 
\newcommand \cht {\chain 2} %Two-element chain
\newcommand \aeqref [1] {(\upshape A\ref{#1})}
\newcommand\set [1]{\{#1\}}
\newcommand \tbf[1] {\textbf{#1}}
\newcommand \nonparallel {\mathrel{\not\kern-1.5pt{\mathop\parallel}}}
\newcommand \red [1] {{\color{red}#1\color{black}}}

\newcommand \nothing [1] {}

\begin{document}
\title
{Permuting 2-uniform tolerances on lattices}

\author[G.\ Cz\'edli]{G\'abor Cz\'edli}
\address{University of Szeged, Bolyai Institute, Szeged,
Aradi v\'ertan\'uk tere 1, Hungary 6720}
\email{czedli@math.u-szeged.hu}
\urladdr{http://www.math.u-szeged.hu/~czedli/}
%\urladdr{http://www.math.u-szeged.hu/\textasciitilde{}czedli/}

\thanks{This research was supported by NFSR of Hungary (OTKA), grant
number K 115518}


%06B10 Lattices (1980-now)  Ideals, congruence relations
%06B99 Lattices (1980-now) None of the above, but in this section 


\dedicatory{Dedicated to the memory of Ivo G. Rosenberg}

\date{September 24, 2019; \red{\emph{always} check the author's homepage for possible updates}}

\begin{abstract}A \emph{$2$-uniform tolerance} on a lattice is a compatible tolerance relation such that all of its blocks are 2-element. We characterize permuting pairs of 2-uniform \emph{tolerances} on lattices of finite length. In particular, any two 2-uniform \emph{congruences} on such a lattice permute.
\end{abstract}

\maketitle

\section{Introduction and result}
In addition to his famous theorem on functional completeness over finite sets, the words ``tolerance'' and ``lattice'' also remind me of Ivo G.\ Rosenberg, since both are common in the title of the present paper and that of  our joint lattice theoretical paper \cite{chczgigr} (co-authored also by I.\ Chajda). A part of my motivation is to keep his memory alive.

This short paper is structured as follows. First, after few necessary definitions, we formulate our main result, Theorem~\ref{thmmain}. Then, still in  this section, we present the rest of our motivation and we point out how the present theorem supersedes its precursor on 2-uniform congruences. Section~\ref{sectionproof} is devoted to the proof of Theorem~\ref{thmmain}.

\subsection*{Definitions and the result}
By a \emph{tolerance} $T$ on a lattice $L$ we mean a reflexive, symmetric, and compatible relation on $L$. The maximal subsets $X$ of $L$ such that $X^2\subseteq T$ are called the \emph{blocks} of $L$. If $T$ is a tolerance such that  each of its blocks consists of exactly two elements, then we call it a \emph{$2$-uniform tolerance} on $L$. As usual, for tolerances $T$ and $S$ on $L$, the \emph{product} $T\circ S$ is defined to be $\{(x,z)$ : there exists a $y\in L$  such that $(x,y)\in T$  and $(y,z)\in S\}$. We say that $T$ and $S$ \emph{permute} if $T\circ S=S\circ T$. Next, assume that $T$, $S$, and $R$ are 2-uniform tolerances on a lattice $L$, and let $u\in L$. Since tolerance blocks are known to be convex sublattices by, say, Cz\'edli~\cite{czglperrho} and since the singleton set $\set{u}$ is not an $R$-block by 2-uniformity, at least one of the following two possibilities  holds:
\begin{enumeratei}
\item\label{enucsgytbxa} there exists a lower cover $v$ of $u$ (in notation, $v\prec u$) such that $\set{v,u}$ is an $R$-block; then $u$ is called an \emph{$R$-top} (element) and $v$ is the \emph{lower $R$-neighbour} of $u$; or
\item\label{enucsgytbxb} there exists an upper cover $w$ of $u$ such that $\set{u,w}$ is an $R$-block; then $u$ is called an $R$-bottom (element) and $w$ is the \emph{upper $R$-neighbour} of $u$.
\end{enumeratei}
Since the $R$-blocks are convex sublattices, it is easy to see that $v$
 in \eqref{enucsgytbxa} is unique, so is $w$ in \eqref{enucsgytbxb} (this explains the definite articles preceding them),  and at least one of $v$ and $w$ exists. 
Of course, the concepts above are meaningful with $T$ or $S$ instead of $R$. 
If $u$ is both a $T$-bottom and an $S$-bottom, then we call it a   \emph{two-fold \ts-bottom}, or a \emph{two-fold} bottom if $T£$ and $S$ are understood. \emph{Two-fold \ts-tops} are defined dually as elements that are simultaneously  $T$-tops and $S$-tops.
Finally, we say that $T$ and $S$ are \emph{amicable} if the following two conditions hold for every $u$ in $L$.
\begin{enumerate}[({\upshape A}1)]
\item\label{eqjnKbxa} If $u$ is a two-fold \ts-top, $u\prec v$ and $(u,v)\in T\cup S$, then $v$ is a two-fold \ts-top. 
\item\label{eqjnKbxb} If $u$ is a two-fold \ts-bottom, $v\prec u$ and $(v,u)\in T\cup S$, then $v$ is a two-fold \ts-bottom.
\end{enumerate}

Note that \aeqref{eqjnKbxa} is the dual of \aeqref{eqjnKbxb}. 
The conjunction of \aeqref{eqjnKbxa} and \aeqref{eqjnKbxb} is easy to imagine as follows: in every component of the graph $(L;T\cup S)$, covers of two-fold tops are two-fold tops and lower covers of two-fold bottoms are two-fold bottoms. 
Now, we are in the position to formulate our
result.

\begin{theorem}\label{thmmain} Let $T$ and $S$ be $2$-uniform tolerances on a lattice $L$ that contains no infinite chain. Then $T$ and $S$ permute if and only if they are amicable.
\end{theorem}

\subsection*{History and further motivation}
Beginning with 
Chajda and Zelinka~\cite{chajdazelinka}, several papers deal with tolerances on lattices; this is exemplified, without seeking completeness with the following list, by  Bandelt~\cite{bandelt}, Chajda~\cite{chajdabook}, Chajda, Cz\'edli, and Rosenberg~\cite{chczgrh}, Cz\'edli~\cite{czglperrho}, Cz\'edli and Gr\"atzer~\cite{czggg}, Grygiel and Radeleczki~\cite{grygrad}, and Kindermann~\cite{kindermann}.
However, the history of the research leading to the present paper began with a problem raised by Gr\"atzer, Quackenbush, and E. T. Schmidt~\cite{GrQuSch}. They asked whether a finite lattice $L$ is necessarily congruence permutable if  any two blocks of each congruence are isomorphic (sub)lattices. Soon thereafter, Kaarli~\cite{kaarli} 
gave an affirmative answer; in fact, he proved even more: if any two blocks of each congruence are of the same size, then the finite lattice in question is congruence permutable. This result was followed by Cz\'edli~\cite{czg2uni1} and \cite{czg2uni2}, which state that in certain finite algebras (including finite lattices), any two 2-uniform congruences permute; a \emph{$2$-uniform congruence} is, of course, a 2-uniform tolerance that happens to be a congruence. Recently, Cz\'edli~\cite{czgdoubling} has applied 2-uniform (and even more general) tolerances in a new construction of modular lattices.

Clearly, any two 2-uniform congruences are amicable. Hence, Theorem~\ref{thmmain} immediately implies the following corollary.

\begin{corollary}\label{corolmHjR}  
If all chains of a lattice $L$ are finite, then any two $2$-uniform congruences of $L$ permute.
\end{corollary}

Although this statement is formulated only for lattices, it supersedes \cite{czg2uni1} and \cite{czg2uni2} in the sense that the lattice in Corollary~\ref{corolmHjR} need not be finite. 
For $n\in\mathbb N:=\set{1,2,3,\dots}$, let  $\chain n$ denote the n-element chain. In order to show an infinite example that belongs to the scope of Corollary~\ref{corolmHjR}, let $K$ be an arbitrary infinite lattice without infinite chains. (For example, we can take all the $\chain n$, $3\leq n\in\mathbb N$, and glue their bottoms into a common bottom and glue their tops into a common top.) Define $L:=\cht\times \cht\times K$, and let $\alpha$ and $\beta$ be the kernel of the first projection and that of the second projection, respectively; then $\alpha$ and $\beta$ are 2-uniform congruences and $L$ has no infinite chain.
Two examples of amicable pairs of 2-uniform tolerances are shown if Figure~\ref{fig}, where the $T$-blocks are given by solid grey ovals while the $S$ blocks by dotted black ones.  Finally, note that neither Theorem~\ref{thmmain}, nor Corollary~\ref{corolmHjR} can be extended to an arbitrary lattice. This is exemplified by the lattice of all integer numbers with the usual ordering; this lattice has exactly two 2-uniform congruences but they do not permute.  

\begin{figure}[htb] 
\centerline
{\includegraphics[scale=1.4]{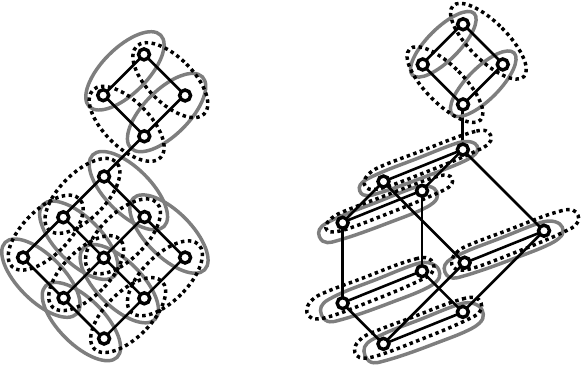}}
\caption{Two examples}
\label{fig}
\end{figure}%

\section{The proof of the  result}\label{sectionproof}

\begin{lemma}\label{lemmalshspcWtlZ}
Let $L$ be a lattice without infinite chains, and let $R$, $T$, and $S$ be $2$-uniform tolerances on $L$. Then, for any $x,y,z,a,b,u\in L$, the following assertions hold.
\begin{enumeratei}
\item\label{lemmalshspcWtlZa}
If $x$ and $y$ are lower $R$-neighbours of $z$, then $x=y$.
\item\label{lemmalshspcWtlZb}
If $(x,y)\in R$, then $x=y$, or $x\prec y$, or $y\prec x$.
\item\label{lemmalshspcWtlZc}
If $a\neq b$, $a$ is the lower $T$-neighbour of $u$, and $b$ is the lower $S$-neighbour of $u$, then  $a\wedge b$ is the lower $S$-neighbour of $a$ and the lower $T$-neighbour of $b$.
\end{enumeratei}
\end{lemma}

Although this lemma is a trivial folkloric  consequence of definitions, we give a short proof for convenience.

\begin{proof} By Zorn's Lemma, any $X\subseteq L$ with $X^2\subseteq R$ extends to a block of $R$, whereby $X^2\subseteq R$ implies that  $|X|\leq 2$. We know from, say, Cz\'edli~\cite{czglperrho} that the blocks of $R$ are convex sublattices. 
If  $x$ and $y$ were distinct lower $R$-neighbours of $z$,
then we would have $(u,z):=(x\wedge y,z\wedge z)\in R$, we could pick a block $B$ of $R$ such that $\set{u,z}\subseteq B$, so $[u,z]\subseteq B$, contradicting $\set{u,z,x,y}\subseteq [u,v]$ and $|B|=2$. 
This shows \eqref{lemmalshspcWtlZa}. 
Part  \eqref{lemmalshspcWtlZb} follows from the trivial fact that it describes the only possibilities where $x$ and $y$ belong to an interval of size at most 2. 
Finally, to prove \eqref{lemmalshspcWtlZc}, assume its  premise. Then  $a$ and $b$ are incomparable (in notation, $a\parallel b$), since both are lower covers of $u$ by \eqref{lemmalshspcWtlZb}. Hence, $\set{a,b}\cap\set{a\wedge b}=\emptyset$. Since 
$(a\wedge b,a)=(a\wedge b,a\wedge u)\in S$, we have that $\set{a\wedge b,a}^2\subseteq S$. Hence, $\set{a\wedge b,a}$ is a block of $S$, and $a\wedge b$ is the lower $S$-neighbour of $a$. Since $(a,S)$ and $(b,T)$ play a symmetric role, \eqref{lemmalshspcWtlZc} follows.
\end{proof}

If $u$ is  a   \emph{two-fold \ts-bottom}, then 
there are two possibilities. Namely, either the upper $T$-neighbour and the upper $S$-neighbour of $u$ are different and we say that $u$ is a \emph{split \ts-bottom}, or these two upper neighbours are the same and we call $u$ an \emph{adherent \ts-bottom}. Dually, if $u$ is a \emph{two-fold \ts-top}, then it is either a \emph{split \ts-top}, or an \emph{adherent \ts-top}, depending on whether its lower neighbours are distinct or equal, respectively. 
Armed with these concepts, we are going to prove the following lemma, which is a bit more  than what the necessity part of Theorem~\ref{thmmain} would require.

\begin{lemma}\label{lemmanckWS}
If  $\,T$ and $S$ are permuting $2$-uniform tolerances on a lattice $L$ without infinite chains, then the following four conditions are satisfied for every $u\in L$.
\begin{enumeratei}%\setcounter{enumi}{4}
\item\label{enucsgytbxe} If $u$ is a split \ts-top, $u\prec v$, and $(u,v)\in T\cup S$, then $v$ is a split \ts-top. 
\item\label{enucsgytbxf} If $u$ is an adherent \ts-top, $u\prec v$, and $(u,v)\in T\cup S$, then $v$ is an adherent \ts-top. 
\item\label{enucsgytbxg} If $u$ is a split \ts-bottom, $v\prec u$, and $(v,u)\in T\cup S$, then $v$ is a split \ts-bottom. 
\item\label{enucsgytbxh} If $u$ is an adherent \ts-bottom, $v\prec u$, and $(v,u)\in T\cup S$, then $v$ is an adherent \ts-bottom.
\end{enumeratei}
\end{lemma}

\begin{proof} With the assumptions of the lemma, in order to prove \eqref{enucsgytbxe}, let $u$ be a split \ts-top, $u\prec v$ and $(u,v)\in T\cup S$. Since $T$ and $S$ play a symmetric role, we can assume that $(u,v)\in T$. The lower $T$-neighbour and the lower $S$-neighbour of $u$ will be denoted by $a$ and $b$, respectively; note that $a\parallel b$, since $a$ and $b$ are distinct lower covers of $u$ by Lemma~\ref{lemmalshspcWtlZ}\eqref{lemmalshspcWtlZb}.
Since $(b,v)\in S\circ T$ and $S\circ T=T\circ S$, there exists an element $c$ such that $(b,c)\in T$ and $(c,v)\in S$.  Observe that $v\not\leq  c$, because otherwise $b<u<v\leq c$ together with  $(b,c)\in T$ would violate Lemma~\ref{lemmalshspcWtlZ}\eqref{lemmalshspcWtlZb}. Hence, again by  \ref{lemmalshspcWtlZ}\eqref{lemmalshspcWtlZb}, $c\prec v$ and $c$ is a lower $S$-neighbour of $v$. If $c\neq u$, then $v$ is a split \ts-top, as required. Hence, it suffices to exclude that $c=u$. For the sake of contradiction, suppose that $c=u$. Then $(b,u)=(b,c)\in T$ indicates that $a$ and $b$ are distinct lower $T$-neighbours of $u$, contradicting  Lemma~\ref{lemmalshspcWtlZ}\eqref{lemmalshspcWtlZa}. This contradiction completes the argument proving  \eqref{enucsgytbxe}. By duality, we conclude the validity of \eqref{enucsgytbxg}.

Next, to prove \eqref{enucsgytbxf}, let $u$ be an adherent \ts-top, $u\prec v$ and $(u,v)\in T\cup S$. Again, we can assume that $(u,v)\in T$. Denote the common lower $T$-neighbour and $S$-neighbour of $u$ by $a$. Since $(a,v)\in S\circ T=T\circ S$, there is an element $c$ such that $(a,c)\in T$ and $(c,v)\in S$. Since both $c\leq a<u<v$ and $a<u<v\leq c$ are excluded by Lemma~\ref{lemmalshspcWtlZ}\eqref{lemmalshspcWtlZb}, 
we obtain from Lemma~\ref{lemmalshspcWtlZ}\eqref{lemmalshspcWtlZb} that $a\prec c\prec v$. As two upper $T$-neighbours of $a$, the elements $u$ and $c$ are the same by the dual of Lemma~\ref{lemmalshspcWtlZ}\eqref{lemmalshspcWtlZa}. Hence, $(u,v)=(c,v)\in S$ shows that $v$ is an adherent \ts-top, as required. This shows the validity of  \eqref{enucsgytbxf}, and  \eqref{enucsgytbxh} follows also by duality.
\end{proof}

\begin{proof}[Proof of Theorem~\ref{thmmain}]
The necessity part follows from Lemma~\ref{lemmanckWS}.
In order to prove the sufficiency part, assume that $T$ and $S$ are amicable. Since $T$ and $S$ play a symmetric role, it suffices to show that $T\circ S\subseteq S\circ T$. So let $(a,b)\in T\circ S$; we need to show that $(a,b)\in S\circ T$. We can assume that $(a,b)\notin T\cup S$, since otherwise the task is trivial.
By the definition of $T\circ S$, there exists an element $u$ such that $(a,u)\in T$ and $(u,b)\in S$. 
Apart from duality, Lemma~\ref{lemmalshspcWtlZ}\eqref{lemmalshspcWtlZb} allows only two cases: either $a\prec u\succ b$, or $a\prec u\prec b$. 
Since  Lemma~\ref{lemmalshspcWtlZ}\eqref{lemmalshspcWtlZc} implies immediately that  $(a,b)\in S\circ T$ in the first case, it suffices to deal only with the second case. That is, $a\prec u\prec b$. Let $x_0:=a$, $x_1:=u$, $x_2:=b$, and define a sequence $x_3,x_4,\dots$ of further elements as follows. If $i$ is even and $x_i$ is a $T$-bottom, then let $x_{i+1}$ be the unique upper $T$-neighbour of $x_i$. If $i$ is odd and $x_i$ is an $S$-bottom, then let $x_{i+1}$ be the unique upper $S$-neighbour of $x_i$. Note that, in addition to the elements $x_i$, $i>2$, the elements $x_1=u$ and $x_2=b$ also obey these rules.
Since $x_2\prec x_3\prec x_4\prec \dots$ but $L$ has  no infinite chain, there is a unique $2\leq n\in \mathbb N$ such that $x_2, x_3,\dots, x_n$ are defined but $x_{n+1}$ is not. There are two (similar) cases depending on the parity of $n$. First, assume that $n$ is even. Since the sequence has terminated with $x_n$, the element $x_{n+1}$ does not exists, that is, $x_n$ is not a $T$-bottom. Hence, $x_n$ is a $T$-top. But $x_n$ is also an $S$-top, whereby $x_n$ is a two-fold \ts-top. The same argument, with the roles of $T$ and $S$ interchanged, shows that $x_n$ is a two-fold \ts-top also in the second case where $n$ is odd. So, $x_n$ is a two-fold \ts-top no matter what the parity of $n$ is. We claim that 
\begin{equation}
\text{$x_{n-2}$ is a two-fold \ts-bottom.}
\label{eqtxtzrhBtrsN}
\end{equation}
 If $x_n$ is an adherent \ts-top, then we obtain from Lemma~\ref{lemmalshspcWtlZ}\eqref{lemmalshspcWtlZa} that 
$x_{n-1}$ is an adherent \ts-bottom, whence $x_{n-2}$ is a two-fold \ts-bottom by \aeqref{eqjnKbxb}, as required. If the two-fold \ts-top $x_n$ is not an adherent one, then it is a split \ts-top, and there are two cases.
If $n$ is even, then $x_{n-1}$ is a lower $S$-neighbour of $x_n$, and $x_n$ has a unique lower $T$-neighbour $c$, which is distinct from $x_{n-1}$. By Lemma~\ref{lemmalshspcWtlZ}\eqref{lemmalshspcWtlZc}, $x_{n-1}\wedge c$ is a lower $T$-neighbour of $x_{n-1}$ and a lower $S$-neighbour of $c$. But $x_{n-2}$ is also a lower $T$-neighbour of $x_{n-1}$, whence Lemma~\ref{lemmalshspcWtlZ}\eqref{lemmalshspcWtlZa} gives that $x_{n-1}\wedge c=x_{n-2}$, and so $x_{n-2}$ is a two-fold (split) \ts-bottom, as required. The same argument works, with $T$ and $S$ interchanged, if $n$ is odd. Thus, \eqref{eqtxtzrhBtrsN} has
been verified. 

Next, we obtain from \aeqref{eqjnKbxb} and \eqref{eqtxtzrhBtrsN}  that $a=x_0$ is also a two-fold \ts-bottom. There are two cases to consider. 
First, assume that $a$ is a split \ts-bottom. Then, in addition that $u$ is an upper $T$-neighbour of $a$, the element $a$ has an upper $S$-neighbour $d$ such that $d\neq u$. By the dual of Lemma~\ref{lemmalshspcWtlZ}\eqref{lemmalshspcWtlZc}, $u\vee d$ is an upper $S$-neighbour of $u$ and an upper $T$-neighbour of $d$. Since $b$ is also  an upper $S$-neighbour of $u$, the dual of  Lemma~\ref{lemmalshspcWtlZ}\eqref{lemmalshspcWtlZa} gives that $u\vee d=b$. Hence, $(a,d)\in S$ and $(d,b)=(d,u\vee d)\in T$ yield that $(a,b)\in S\circ T$, as required. 

Second, assume that  $a$ is an adherent \ts-bottom. Then $u=x_1$ is an (adherent) two-fold \ts-top. Applying \aeqref{eqjnKbxa}, we have that  $b=x_2$ is also a  two-fold \ts-top. Hence, $b$ has a unique lower $T$-neighbour $e$. We claim that $e=u$; for the sake of contradiction, suppose that $u\neq e$.  Applying Lemma~\ref{lemmalshspcWtlZ}\eqref{lemmalshspcWtlZc}, it follows that $u\wedge e$ is a lower $T$-neighbour of $u$. But $a$ is also 
 a lower $T$-neighbour of $u$, whereby  Lemma~\ref{lemmalshspcWtlZ}\eqref{lemmalshspcWtlZa} give that $u\wedge e=a$. On the other hand, Lemma~\ref{lemmalshspcWtlZ}\eqref{lemmalshspcWtlZc} also gives that $u\wedge e=a$ is a lower $S$-neighbour of $e$. Hence, $a$ has two distinct upper $S$-neighbours, $u$ and $e$, which contradicts the dual of Lemma~\ref{lemmalshspcWtlZ}\eqref{lemmalshspcWtlZa}. This contradiction shows that $e=u$.  Armed with this equality, 
 $(a,u)\in S$ and $(u,b)=(e,b)\in T$, and the required $(a,b)\in S\circ T$ follows. We have shown that $T\circ S\subseteq S\circ T$, and the proof of Theorem~\ref{thmmain} is complete.
\end{proof}

\end{document}